\documentclass[12pt,reqno]{amsart}
\usepackage{amssymb,amsfonts,latexsym,amstext,epsfig}

\usepackage[left=2.5cm,top=2.5cm,right=2.8cm,bottom=2.5cm]{geometry}

\numberwithin{equation}{section}
\newtheorem{theorem}{Theorem}
\newtheorem{corollary}[theorem]{Corollary}
\newtheorem{lemma}[theorem]{Lemma}

\newtheorem{claim}{Claim}
\newcommand{\Nnn}{{\mathbb N}}

\newcommand{\Rrr}{{\mathbb R}}

\begin{document}
\title{Infinite random geometric graphs}
\author{Anthony Bonato}
\address{Department of Mathematics\\
Ryerson University\\
Toronto, ON\\
Canada, M5B 2K3} \email{abonato@ryerson.ca}
\author{Jeannette Janssen}
\address{Department of Mathematics and Statistics\\
Dalhousie University\\
Halifax, NS\\
Canada, B3H 3J5}
\email{janssen@mathstat.dal.ca}
\thanks{The authors gratefully acknowledge support from NSERC and MITACS}

\begin{abstract}
We introduce a new class of countably infinite random geometric graphs, whose vertices $V$ are points in a metric space, and
vertices are adjacent independently with probability $p\in (0,1)$ if the metric distance between the vertices is below a given
threshold.
If $V$ is a countable dense set in $\mathbb{R}^n$ equipped
with the metric derived from the $L_{\infty}$-norm, then it is shown that with probability $1$
such infinite random geometric
graphs have a unique isomorphism type. The isomorphism type, which we call $GR_n,$ is characterized by a
geometric analogue of the existentially closed adjacency
property, and we give a deterministic construction of $GR_n.$
In contrast, we show that infinite random geometric graphs in
$\mathbb{R}^2$ with the Euclidean metric are not necessarily isomorphic.
\end{abstract}

\maketitle

\section{Introduction\label{intro}}

The last decade has seen the emergence of the study of large-scale complex networks such as the web graph consisting of web pages and the links between them, the friendship network on Facebook,
and networks of interacting proteins in a cell. Several new random graph models were proposed for such networks, and existing models were modified in order to fit the
data gathered from these real-life networks. See the books~\cite{bonato,clbook} for surveys of such models. A recent trend in
stochastic graph modelling is the study of {\sl geometric graph
  models}. In geometric graph models, vertices are embedded in a
metric space, and the formation of edges is influenced by the relative
position of the vertices in this space. Geometric graph models have found
applications in modelling wireless networks (see \cite{frieze,goel}),
and in modelling the web graph and other complex networks
(\cite{bon,FFV}). In real-world networks, the underlying metric space is a representation of the \lq\lq
hidden reality'' that leads to the formation of edges. Thus, for the
World Wide Web, web pages are embedded in a
high dimensional \emph{topic space}, where pages that are positioned
close together in the space contain similar content.

The graph model we study is a variation on the {\sl random geometric graph},
where vertices are chosen at random according to
a given probability distribution from a given metric space, and two
vertices are adjacent if the distance between the two vertices is no
larger than some fixed real number. Random geometric graphs have been
studied extensively in their own right (see, for example,
\cite{bol,ellis} and the book \cite{mpen}).

Analysis of stochastic graph models usually focusses on asymptotic
results, which hold for cases where the number of vertices is
sufficiently large.
An alternative approach is to study the {\sl
  infinite limit}; that is, the infinite graph that results when the number of
vertices reaches infinity. Studying the infinite limit is a well-known
tool for studying scientific models, and it can help to recognize large-scale
structure and long-term behaviour (see \cite{bj,klein1}). In this paper, we study the infinite limit of a
geometric graph model that is a geometric extension
of the classic random graph model $G(\mathbb{N},p)$.

One of the most studied examples of an infinite limit graph arising from a
stochastic model is the infinite random graph. The probability space
$G(\mathbb{N},p)$ consists of graphs with vertices $\mathbb{N},$ so that
each distinct pair of integers is adjacent independently with a fixed
probability
$ p\in (0,1).$ Erd\H{o}s and R\'{e}nyi \cite{er} discovered that with probability $1,$
all $G\in G(\mathbb{N},p)$ are isomorphic. A graph $G$ is \emph{existentially closed} (or \emph{e.c.}) if for all finite
disjoint sets of vertices $A$ and $B$ (one of which may be empty),
there is
a vertex $z\notin A\cup B$ adjacent to all vertices of $A$ and to no vertex of
$B.$ We say that $z$ is \emph{correctly joined}
to $A$ and $B.$ The unique isomorphism type of
countably infinite e.c.\ graph is named the \emph{infinite random graph}, or
the \emph{Rado} graph, and is written $R.$ See the surveys
\cite{cam,cam1} for additional background on $R.$

We now introduce the geometric graph model on which this paper is based.
Consider a metric space $S$
with distance function $$d:S \times S \rightarrow \mathbb{R},$$
$\delta\in \Rrr^+$,
a countable subset $V$ of $S$, and $p\in (0,1)$. The {\sl Local Area
  Random Graph $\mathrm{LARG}(V,\delta,p)$} has vertices $V,$ and
for each pair of vertices $u$ and $v$ with $d(u,v)<\delta$, an edge is
added independently with probability $p$. Note that $V$ may be either finite or infinite. The LARG model generalizes
well-known classes of random graphs.
For example, special cases of the $\mathrm{LARG}$ model include the
random geometric graphs (where $p=1$), and the binomial random graph
$G(n,p)$ (where $S$ has finite diameter $d$, and $\delta\ge d$).

In the case $V$ is infinite, our goals are to investigate what adjacency properties are satisfied by
graphs generated by the LARG model, and determine when the model generates a
unique isomorphism type of countable graph. We prove that with probability $1$, graphs in
$\mathrm{LARG}(V,\delta,p)$ satisfy a certain adjacency property which
is a metric analogue of the e.c.\ property; see
Theorem~\ref{random}. The so-called geometric e.c.\ property requires
that vertices $z$ correctly joined to $A$ and $B$ may be found as
close as we like to points in $V.$ Explicit examples of graphs with
the geometric e.c.\ property are given in Theorem~\ref{explicit}. For
metric spaces such as $\mathbb{R}^n$ with the $L_{\infty}$-metric, the
geometric e.c.\ property gives rise to a unique isomorphism type of
graph; see Theorems~\ref{thm:1D}, \ref{thm:1Disometry}, and
\ref{thm:nD}. The main tool here is a generalization of isometry
called a step-isometry. Our results are sensitive to the metric
used. Non-isomorphism results for other metrics are given in
Section~\ref{none}. In particular, we show in Theorem~\ref{non2} that in $\mathbb{R}^2$ with the Euclidean metric, there
exist non-isomorphic geometric e.c.\ graphs.

All graphs considered are simple, undirected, and countable unless
otherwise stated. If $S$ is a set of vertices in $G$, then we use
the notation $G[S]$ for the subgraph of $G$ induced by $S.$ We use the notation $G\le H$ if $G$ is an induced subgraph of $H.$ We refer to an isomorphism type as \emph{isotype}, and denote isomorphic graphs by $G\cong H.$ Given a metric space $S$ with distance function $d$, define the (open) \emph{ball of radius
$\delta$ around $x$} by
\[
B_\delta(x)=\{ u\in S:d(u,x)< \delta\}.
\]
We will sometimes just refer to $B_{\delta}(x)$ as a \emph{ball}.
A subset $V$ is \emph{dense} in $S$ if for every point $x\in S$, every ball
around $x$ contains at least one point from $V$.
We refer to $u \in S$ as points
or vertices, depending on the context. Throughout, let $\mathbb{N}$ denote the non-negative natural numbers, and $\mathbb{N}^+$ denote the positive natural numbers. For a reference on graph theory
the reader is directed to \cite{west}, while \cite{bryant} is a
reference on metric spaces.

\section{Geometrically e.c.~graphs}\label{geo}

As noted in the introduction, the unique isotype of the infinite random graph $R$ is
characterized by the e.c.~property. In this section we define a
geometric analogue of this property. As will be demonstrated in Section~3, this property characterizes the unique isotype of graphs
obtained from countable dense sets in $\Rrr^n$, provided we consider a particular
metric.

Let $G=(V,E)$ be a graph whose vertices are points in the metric space $S$ with metric $d.$ The graph $G$ is
\emph{geometrically e.c}.\ \emph{at level $\delta$} (or $\delta$-\emph{g.e.c.}) if for all
$\delta'$ so that $0<\delta'<\delta$, for all $x\in V$, and for all
disjoint finite sets $A$ and $B$ so that $A\cup B\in B_\delta(x)$, there
exists a vertex $z\not\in A\cup B\cup\{x\}$ so that
\begin{itemize}
\item[($i$)] $z$ is correctly
joined to $A$ and $B$,
\item[($ii$)] for all $u\in A\cup B$, $d(u,z)<\delta$, and
\item[($iii$)]
$d(x,z)<\delta'$.
\end{itemize}
This definition implies that $V$ is dense in itself.
Also, if $G$ is $\delta$-e.c.,
then $G$ is $\delta'$-e.c. for any $\delta'<\delta$.

The geometrically e.c.~property bears clear
similarities with the
e.c.\ property. The important differences are that a
correctly joined vertex must exist only for
sets $A$ and $B$ which are contained in an open ball with radius
$\delta$ and centre $x$, and it must be possible to choose the vertex
$z$ correctly joined to $A$ and $B$ arbitrarily close to $x$. See Figure~1.

\begin{figure}[h]
\begin{center}
\epsfig{figure=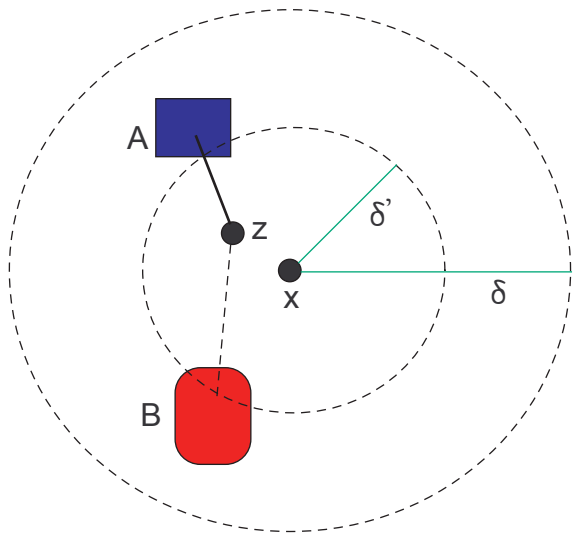} \caption{The $\delta$-g.e.c.\ property.}
\end{center}
\end{figure}

\begin{theorem}\label{random}
Let $(S,d)$ be a metric space and $V$ a countable subset of $S$ which
is dense in itself. If $\delta>0$ and $p\in (0,1),$ then with probability $1,$
$\mathrm{LARG}(V,\delta,p)$ is $\delta$-e.c.
\end{theorem}

\begin{proof}
Fix $x\in V$, disjoint
finite subsets $A$ and $B$
in $B_{\delta}(x)\cap (V\setminus\{x\})$, and $0<\delta'<\delta$. Let
\[
\beta=\max\{d(x,v):v\in A\cap B\}.
\]
Then $\beta <\delta$. Let
  $\epsilon=\min\{ \delta-\beta,\delta'\}$.
Consider the set $Z=B_{\epsilon}(x)\cap V.$ Note that $\epsilon$ is
chosen so that for any $z\in Z$, $d(z,x)<\delta'$, and for all $u\in
A\cup B$,
\[
d(u,z)<d(u,x)+d(x,z)<\beta+\epsilon\le \delta.
\]

For any graph $G$ in $\mathrm{LARG}(V,\delta,p)$, the probability that
any vertex $z\in Z$ is correctly joined to $A$ and
$B$ equals $p^{|A|}(1-p)^{|B|}$. The probability that no vertex in $Z$
is correctly joined to $A$ and $B$ equals
\[
P=\prod_{z\in Z} 1-\left(p^{|A|}(1-p)^{|B|}\right).
\]
Since $V$ is dense in itself, $Z$ contains infinitely many points; hence, $P=0.$
As there are only countably many choices for $x$, $A$, and $B$, and a
countable union of measure $0$ sets is measure $0$, the proof
follows. \end{proof}

A graph $G=(V,E)$ whose vertices are points in the metric space $(S,d)$ has \emph{threshold} $\delta$ if for
all edges $uv\in E$, $d(u,v)<\delta$. A graph that is geometrically
g.e.c.~at level $\delta$ and has threshold $\delta$ is called a {\sl
geometric $\delta$-graph. } By definition, a graph $G$ generated by $\mathrm{LARG}(V,\delta,p)$ has threshold $\delta$,
and, if $V$ is countable and dense in itself,
$G$ is a geometric $\delta$-graph. Thus, this random graph model generates
geometric $\delta$-graphs.

We construct geometric $\delta$-graphs
deterministically as follows. Given $\delta>0$, a countable set
$V$ which is dense in itself,
and a linear ordering $\sigma:\Nnn\rightarrow V$ of $V$, define
$GR(V,\delta,\sigma)$ as the limit of a chain of finite graphs
$R_t$, where $R_t\leq R_{t+1}$ for any $t>1$, and
$\{\sigma(i):1\leq i\leq t\}\subseteq V(R_t)$.
Let $R_1$ be the trivial
graph with vertex set $\sigma(1)$.  Assume that $R_t$ is defined and
$\{\sigma(i):1\leq i\leq t\}\subseteq V(R_t).$

We now define $R_{t+1}.$
Enumerate all pairs $(A,x)$ so that $A\subseteq
V(R_t)$ and $x\in V(R_t)\setminus A$ so that $A\subseteq B_\delta(x)$, via a
lexicographic ordering based on $\sigma$. For each pair $(A,x)$, in
order, choose
$z=z_{A,x}$ to be the least index point in $V$ (according to $\sigma$) such
that $z$ has not been chosen for any previous pairs $(A,x)$,
\begin{equation}\label{ca}
B_\delta(z)\cap V(R_t) = B_\delta(x)\cap V(R_t),
\end{equation}
 and
$$d(z,x)<\min\{1/t,\delta\}.$$
Note that such a vertex exists as $V$ is dense and $R_t$ is finite. Join $z$ to all
vertices in $A$ and to no other vertices of $R_t$. If necessary, add
$\sigma(t+1)$ as an
isolated vertex to form the graph $R_{t+1}.$ Observe that by (\ref{ca}), $GR(V,\delta,\sigma)$ is a
$\delta$-threshold graph.

\begin{theorem}\label{explicit}
The graph $GR(V,\delta,\sigma)$ is
$\delta$-g.e.c.
\end{theorem}

\begin{proof}
To show that $GR(V,\delta,\sigma)=(V,E)$ is $\delta$-g.e.c., choose
$0<\delta'<\delta$, a vertex $x\in V,$ and disjoint sets $A,B\subseteq
V\setminus\{x\}$ so that $A\cup B\subseteq B_\delta(x)$. Let $t>0$ be
chosen so that $A\cup B\cup \{x\}\subseteq V(R_t)$ and
$1/t<\delta'$. Let $z=z_{A,x}$ be the vertex in $R_{t+1}$ added to
$R_t$ to extend
$(A,x)$. Then $z$ is
correctly joined to $A$ and $B$ and $d(x,z)<\delta'$.
\end{proof}

It should be emphasized that we do not claim that the graphs
$GR(V,\delta,\sigma)$ are the unique isotypes of
$\delta$-g.e.c.~graphs with vertex set $V.$ The theme of when two
$\delta$-g.e.c.~graphs are isomorphic will be explored in the next two
sections.

Balls with radius $\delta$ in $\delta$-g.e.c.~graphs contain copies of
$R$, and hence, contain isomorphic copies of all countable graphs.

\begin{theorem}\label{subset}
Let $U\subseteq S$ be so that
$U\subseteq B_{\delta}(x)$ for
some $x\in U.$ Then a $\delta$-g.e.c.~graph with vertex set $U$ is
e.c., and so is isomorphic to $R$.
\end{theorem}

\begin{proof}
Let $G$ be a graph with vertex set $U$, where $U$ is as stated. Assume
that $G$ is $\delta$-g.e.c.
Let $A$ and $B$ be any pair of disjoint, finite subsets of $U\subseteq
B_\delta(x)$, and
let $0<\delta'<\delta$. Then by
the $\delta$-g.e.c.~condition there exists a vertex
$z\in B_{\delta'}(x)\cap U$ so that $z$ is correctly joined to $A$ and $B$.
\end{proof}

The converse of Theorem~\ref{subset} is false, in general. For
example, consider the metric space $(\mathbb{R},d)$, where $d$ is the
Euclidean metric, and let $\delta =1$. Fix $U$ an infinite
clique in $R$, and let $U'=V(R)\setminus U$. Embed the vertices of $U$
in $\Rrr$ so that they form a set that is dense in $B_{1/2}(0)$. Embed
the vertices of $U'$ so that they form a set that is dense in
$B_{1}(0)\setminus B_{1/2}(0)$. Now choose $y\in U$ so that
$d(0,y)<1/4$, and let $A=\emptyset$, and $B=\{b\}$, where $b\in
U\setminus \{y\}$. Let $\delta'=1/4$. Note that $A\cap B\subseteq
B_{1}(y)$. The embedding of the vertices of $R$ is such that all
vertices in $B_{\delta'}(y)$ are in $U$, so they are all adjacent to
$b$. Thus, $B_{\delta'}(y)$ does not contain any vertex correctly
joined to $A$ and $B$, and thus, this embedding of $R$ is not
$\delta$-g.e.c.

We finish with the following theorem which shows that there exists a
close relationship between graph distance and metric distance in any
graph that is $\delta$-g.e.c. We denote the closure of set $V$ in $S$ by
$\overline{V}$. The set $W$ is \emph{convex} if for every pair of points
$x$ and $y$ in $W,$ there exists a point $z$ such
that $$d(x,z)+d(z,y)=d(x,z).$$

\begin{theorem}
\label{mot}
Let $G=(V,E)$ be geometric $\delta$-graph, and let $\overline{V}$ be convex.
Let $u,v\in V$ so that $d(u,v)>\delta$. Then the graph distance
between $u$ and $v$ in $G$ equals $\lfloor d(u,v)/\delta\rfloor +1$.
\end{theorem}

Theorem~\ref{mot} directly leads to the following corollary, which
supplies motivation for proofs of the isomorphism results in the next
section, and will be used to prove the non-isomorphism results of the final
section.

\begin{corollary}\label{useful}
  If $\overline{V}$ and $\overline{W}$ are convex, and there is a
  $\delta$-g.e.c.\ graph with vertices $V$ and a
  $\gamma$-g.e.c.\ graph with vertex set $W$ which are isomorphic via $f$, then for
  every pair of vertices $u,v\in V$, $$
\lfloor d(u,v)/\delta\rfloor=\lfloor d(f(u),f(v))/\gamma\rfloor.
$$
\end{corollary}

We supply a generalization of isometry, motivated by Theorem~\ref{mot}
and Corollary~\ref{useful}. Given metric spaces $(S,d_S)$ and $(T,d_T)$, sets $V\subseteq S$ and
$W\subseteq T$, and positive real numbers
$\delta$ and $\gamma$, a \emph{step-isometry at level
  $(\delta,\gamma)$} from $V$ to $W$ is a surjective map
$f:V\rightarrow W$ with the property that for every pair of vertices
$u,v\in V,$
\[
\lfloor d_S(u,v)/\delta\rfloor=\lfloor d_W(f(u),f(v))/\gamma\rfloor.
\]
Every isometry is a step-isometry, but the converse is false, in
general. For example, consider $\Rrr$ with the Euclidean metric, let
$\delta = \gamma =1$, and
let $V=[0,1)$ and $W=[0,0.5)$. Then $f: V\rightarrow W$ given by $f(x)=x/2$ is a
step-isometry, but is not an isometry.

\begin{proof}[Proof of Theorem~\ref{mot}]
Let $u,v\in V$. Let $k=\lfloor d(u,v)/\delta\rfloor +1$. By assumption,
$k\geq 2$. Note that the choice of $k$ supplies that $$(k-1)\delta \leq d(u,v) <
k\delta.$$ Let $\ell$ be the graph distance of $u$ and $v,$ and note that $\ell >1$ since $G$ is a $\delta$-threshold graph.

To show that $\ell\geq k$, let $v_0v_1\cdots v_\ell$, where
$v_0=u$, $v_\ell=v$, be a shortest path in $G$ from $u$ to $v$.  Since
$G$ has threshold $\delta$, $d(v_{i-1},v_{i})<\delta$ for
$i=1,\dots,\ell $. Therefore,
\begin{eqnarray*}
(k-1)\delta&\leq &d(u,v) \\
&\leq &\sum_{i=1}^{\ell} d(v_{i-1},v_{i})\\
 &< &\ell\delta,
\end{eqnarray*}
and so $\ell \geq k$.

Next, we show how to construct a path of length $k$ from $u$ to $v$ in
$G$, which will prove that $\ell\leq k$. Let
$\epsilon=(k\delta-d(u,v))/k$, so $d(u,v)=k(\delta-\epsilon)$.

The set $\overline{V}$ is convex, so for every pair of vertices $x,y$,
there exists a point $z\in \overline{V}$ so that $d(x,z)+d(z,y)=d(x,y)$. Using
this property, we can obtain a sequence of points in $\overline{V}$
between $u$ and $v$ whose successive distances add up to $d(u,v)$, and
which are at most $\epsilon/4$ apart. We  can then choose vertices
$x_1,\dots ,x_{k-1}$ from this sequence so that
$d(x_i,x_{i+1})<\delta-3\epsilon/4$ for $i=0,\dots, k-1$, where
$x_0=u$ and $x_k=v$. For $1\leq i<k$, we may then find $w_i\in V$ so
that $d(w_i,x_i)<\epsilon/8$. Letting $w_0=u$ and $w_k=v$, we have
that for $i=0,\dots, k-1$,
\begin{eqnarray}
d(w_i,w_{i+1})&\leq &d(w_i,x_i)+d(x_i,x_{i+1})+d(x_{i+1},w_{i+1}) \nonumber \\
&<&\delta-3\epsilon/4+2\epsilon/8 \nonumber \\
&<&\delta-\epsilon/2. \label{cha}
\end{eqnarray}

Let $v_0=w_0=u$. Now we successively apply the $\delta$-e.c.\ property to choose $v_i\in V$ so that
\begin{itemize}
\item[($i$)] $d(v_i,w_i)<\epsilon/2$, and
\item[($ii$)] $v_{i}$ is adjacent to $v_{i-1}$ in $G$.
\end{itemize}
It then follows that $v_0v_1\cdots v_k$ is the desired path of length $k.$
More precisely, fix $i$, $1\leq i<k-1$, and assume $v_{i-1}$ exists so that item $(i)$ holds. Note that ($i$) and (\ref{cha}) implies that
$$d(w_{i},v_{i-1})\leq d(w_{i-1},w_{i})+d(w_{i-1},v_{i-1})<\delta.$$
Therefore, $v_{i-1}\in B_{\delta}(w_i)$, and so we can find a vertex
$v_i$ in $B_{\epsilon/2}(w_i)$ which is adjacent to $v_{i-1}$.
To choose the vertex $v_{k-1}$, let $v_k=w_k=v$. By the same
argument as before, $d(w_{k-1},v_{k-2})<\delta$. Since $w_k=v_k$,
$d(w_{k-1},v_k)< \delta$. So $\{ v_{k-2},v_k\}\subseteq
B_\delta(w_{k-1})$. Therefore, there exist a vertex $v_{k-1}$ which is
adjacent to $v_{k-2}$ and $v_k=v$.
\end{proof}

\section{Isomorphism results}\label{isor}

In this section we consider metric spaces where the geometric
e.c.\ property gives a unique isotype of graph. We work in the space $\Rrr$ with the usual metric defined by $d(x,y)=|x-y|$. (We will
not mention this explicitly unless there is room for confusion.) The first result of the
section---which serves as the template for more general
results---is the following.
\begin{theorem}
\label{thm:1D}
Let $V$ and $W$ be two countable dense subsets of $\Rrr$, and let
$\delta,\gamma>0$.
If $G$ is a geometric $\delta$-graph with vertex set $V$ and $H$
is a geometric $\gamma$-graph with vertex set $W,$ then $G\cong H.$
\end{theorem}
The proof of the theorem (and others analogous to it in this section) build up an isomorphism as a
step-isometry. In the proofs we use the following alternative characterization of
step-isometries. Fix $\delta>0$ and $v_0\in \Rrr$. Each
 $v\in \Rrr$ may be
uniquely represented as
$$v=v_0+q(v)\delta+r(v),$$ where
$q(v)=\lfloor
(v-v_0)/\delta\rfloor $ and  $0\leq
r(v)<\delta$. In this representation, we will refer to $\delta$ as the
{\sl offset}, and to $v_0$ as the {\sl anchor}. We will omit to state the anchor and offset explicitly
wherever it is clear from the context. The term $r(v)$ is
called the {\sl representative} of $v$ and $q(v)$ the
{\sl quotient}.

\begin{lemma}
\label{lemma:isometry_conditions}
Let $V$ and $W$ be subsets of $\Rrr$, and let $\delta$ and $\gamma$ be
two non-negative real numbers. A surjective function $f:V\rightarrow
W$ is a step-isometry at level $(\delta,\gamma)$ if and only if the following two
conditions hold.
\begin{enumerate}
\item For every $u,v\in V$, $r(u)\leq r(v)$ if and only if $r(f(u))\leq r(f(v)).$
\item For every $u\in V$, $q(u)=q(f(u))$,
\end{enumerate}
where the representation of elements of $V$ has offset $\delta$ and
that of $W$ has offset $\gamma$, and the anchor of the representation
of $W$ is the image of the anchor of the representation of $V$ under $f$.
\end{lemma}

\begin{proof}
Assume first that items (1) and (2) hold, and fix $u,v\in V$.  Let
$u'=f(u)$ and $v'=f(v)$. Assume without loss of generality that $v>u$;
by hypothesis, this implies that $v'>u'$.  Then
$d(u,v)=v-u=(v-v_0)-(u-v_0)$, and so
\[
d(u,v)=(v-v_0)-(u-u_0)=(q(v)-q(u))\delta+(r(v)-r(u)).
\]
Hence,
\[
\lfloor d(u,v)/\delta\rfloor= q(v)-q(u)-s,
\]
 where $s=0$ if $r(v)\geq
r(u)$, and $s=1$ otherwise. Similarly,
$$
\lfloor d(u',v')/\gamma\rfloor=q(v')-q(u')-s',
$$
where $s'=0$ if $r(v')\geq
r(u')$, and $s'=1$ otherwise. By hypothesis, we have that $s=s'$ and
$q(v)-q(u)=q(v')-q(u')$.
It follows that $f$ is a step-isometry at
level $(\delta,\gamma)$.

Now assume that $f$ is a step-isometry at level
$(\delta,\gamma)$. Let $v_0\in V$ and $w_0\in W$ be so that
$f(v_0)=w_0$ and consider the representations of elements of $V$ and $W$ with
offsets $\delta$ and $\gamma$, and anchors $v_0$ and $w_0$,
respectively.

Condition (2) follows immediately from the definition of
step-isometry. For the proof of (1), fix any $u,v\in V$, and let
$u'=f(u)$ and $v'=f(v)$. Then
$$\lfloor d(u,v)/\delta\rfloor=
\lfloor
q(v)-q(u)+(r(v)-r(u))/\delta \rfloor = q(v)-q(u)-s,$$ where $s=-\lfloor
(r(v)-r(u))/\delta \rfloor$ ; similarly, $\lfloor d(u',v')/\gamma\rfloor =
q(v')-q(u')-s'$, where  $s'=-\lfloor
(r(v')-r(u'))/\delta \rfloor$. Since $f$ is a step-isometry, $s=s'$. If
$s=s'=0$, then both
$r(v)\geq r(u)$ and $r(v')\geq r(u')$; if $s=s'=1$, then
$r(v)< r(u)$ and $r(v')< r(u')$.
Thus item (1)
holds.
\end{proof}

\begin{proof}[Proof of Theorem~\ref{thm:1D}] The proof follows using a variant
of the back-and-forth method (used to show that $R$ is the unique isotype of e.c.\ graph).
Let $V=\{v_i:i\geq 0\}$ and $W=\{w_i:i\geq 0\}$. For $i\ge 0$, we
inductively construct a sequence of pairs of sets $(V_i,W_i)$ and
isomorphisms $f_i:G[V_i]\rightarrow H[W_i]$, so that for all $i\geq
1$, $v_i\in V_i$, $w_i\in W_i$, $V_{i}\subseteq V_{i+1}$ and
$W_i\subseteq W_{i+1}$, and $f_{i+1}$ extends $f_i$. It follows
that $$\bigcup_{i\in \mathbb{N}} f_i: G \rightarrow H$$ is an
isomorphism.  As an additional induction hypothesis we require that
$f_i$ is a step-isometry from $V_i$ to $W_i$ at level
$(\delta,\gamma)$.
Specifically, we maintain conditions in
items (1) and (2) from Lemma~\ref{lemma:isometry_conditions}, where
the representation of elements of $V$ has offset $\delta$ and anchor $v_0$, and
the representation of elements of $W$ has offset $\gamma$ and anchor $w_0$.

Let $V_0 = \{ v_0 \}$, $W_0 = \{ w_0\},$ and define
$f_0$ by $f_0(v_0) =w_0.$ Then $q(v_0)=q(w_0)=0$ and
$r(v_0)=r(v_{0}')=0$, so the base case of the induction
follows. For the induction step, fix $i\geq 0$. To construct $f_{i+1}$
from $f_{i}$ we first \emph{go forth} by finding an image of
$v_{i+1}$. In the following, $f$ refers to $f_{i}$ and $v=v_{i+1}$.

Define
\begin{eqnarray*}
a&=&\max\{r(f(u)):u\in V_i\mbox{ and }r(u)\leq r(v)\},\\
b&=&\min\{r(f(u)):u\in V_i\mbox{ and } r(u)> r(v)\}.\\
\end{eqnarray*}
We claim that $a<b$. Namely, let $u_a$ and $u_b$ be the elements in
$V_i$ for which the maximum and minimum that define $a$
and $b$ are attained, respectively. Thus, $r(f(u_a))=a$ and $r(f(u_b))=b$. By
definition, $r(u_a)\leq r(v)<r(u_b)$. By the induction
hypothesis (specifically, item (1) from Lemma
\ref{lemma:isometry_conditions}), this implies that
$a=r(f(u_a))<r(f(u_b))=b$.

In order to maintain the induction hypothesis, $r(f(v))$ must lie in
$[a,b)$, and $q(f(v))$ must equal to $q(v)$. Let $k=q(v)$, and
  consider the interval
$$
I=(k\gamma +a,k\gamma +b).
$$
Any vertex in $I$
  will qualify as a candidate for $f(v),$ so that $f_{i+1}$ is a
  step-isometry at level $(\delta, \gamma).$ We must then find a
  vertex in $I$ that will also guarantee that $f$ is an isomorphism, by
  making sure it has the correct neighbours. For this, we apply the
  $\gamma$-g.e.c.~condition of $H$.

In order to apply the
  $\gamma$-g.e.c.~condition, we need to ensure that the images of all
  neighbours of $v$ in $V_i$ lie in a $\gamma$-ball. Since $G$ has
  threshold $\delta$, we consider all vertices of $V_i$ that lie in a
  $\delta$-ball around $v$.
Let $Y=B_\delta(v)\cap V_i$, and fix $x\in I\cap W$. Such a vertex $x$
exists since $W$ is dense in $\Rrr$. By definition of $I$,
$q(x)=k$. We claim that
\begin{equation}\label{moo}
f(Y)\subseteq B_{\gamma}(x).
\end{equation}
To prove this, let $u\in Y.$ Since $q(v)=k$ and $d(u,v)<\delta$, it
follows that $|q(u)-k|\leq 1$. Hence, $q(u)$ is one of $k,$ $k-1,$ or
$k+1.$

If $q(u)=k,$ then $q(f(u))=k$ by induction hypothesis, so
$d(f(u),x)<\gamma$. If $q(u)=k-1$, then $r(u)>r(v)$, so $r(f(u))>b$ by
definition of $b$. Hence,
\begin{eqnarray*}
d(f(u),x) & = & x-f(u) \\
& < &k\gamma+b-(k-1)\gamma-r(f(u))\\
& < &\gamma.
\end{eqnarray*}
The final case is when $q(u)=k+1.$ Then $r(f(u))\leq a$, so we have that
\begin{eqnarray*}
d(f(u),x) & = & f(u)-x \\
&<&(k+1)\gamma+r(f(u))-k\gamma-a\\
&\leq & \gamma.
\end{eqnarray*} In all cases, $f(u)\in B_{\gamma}(x),$ and (\ref{moo}) follows.

Since $G$ has threshold $\delta$, $N(v)\cap V_i\subseteq Y$. Now let
$A=f(N(v)\cap V_i)$ and $B=(W_i\cap B_\delta(x))\setminus A$. Then
$A\cap B\subseteq B_{\gamma}(x)\cap W_i$. Let $\epsilon>0$ be chosen
such that $B_\epsilon(x)\subseteq I$.
We now use the $\gamma$-e.c.~property of $H$ to find a point $z\in
B_\epsilon(x)$ which is adjacent to all vertices in $A$ and no other
vertices of (the finite set) $W_i$. Thus, we can add $z$ to $W_i$ to
form $W_{i+1}$ and add $v$ to $V_i$ to form $V_{i+1}$, and set
$f_{i+1}(v)=z$. Observe that $f_{i+1}$ is an isomorphism.

To finish the induction step, if $w_{i+1}\not\in W_{i+1}$ then we may
\emph{go back}, by finding an image $z=f^{-1}_{i+1}(w_{i+1})$  in an
analogous fashion. We then add $z$ to $V_{i+1}$, and maintain that
$f_{i+1}$ is an isomorphism.
\end{proof}

The proof of the following corollary is now immediate.

\begin{corollary}
\label{cor}
For all countable dense subsets $V$ of $\Rrr$, $\delta >0$, and $p\in
(0,1),$  with probability $1$, there is a unique isotype of graph,
written $GR_1,$ in $\mathrm{LARG}(V,\delta,p)$.
\end{corollary}
The isomorphism type of $GR_1$ does not depend on the
choices of $V$, $\delta$ or $p$; moreover, the same result holds for
any 1-dimensional normed vector space with the metric derived from the norm.
For this reason, we name
$GR_1$ the \emph{infinite random geometric graph of
  dimension 1}.  Note that $GR_1$ has infinite
diameter (unlike $R$, which has diameter $2$). Note that, for any countable set $V\subseteq \Rrr$, any
ordering $\sigma$ of $v$,
and any real $\delta>0$, the deterministic construction process
$R(V,\sigma,\delta)$ described in the previous section gives explicit representations of $GR_1$.

We may extend Theorem~\ref{thm:1D} to sets that are not necessarily
dense in all $\Rrr$. The only additional condition required is that
there exists a step-isometry between the two sets. For example, consider the rational intervals $V=[a,b)$ and $W=[a',b')$, where
    $\lfloor (b-a)/\delta\rfloor=\lfloor
    (b'-a')\gamma\rfloor$. Consider the
    bijective map $f:V\rightarrow W$ defined by
\[
f(x)=\left\{ \begin{array}{ll} a'+q(x)\gamma +r(b')r(x)/r(b) &
\mbox{    if } r(x)\leq r(b)\\
a'+(q(x)+1)\gamma +(\gamma - r(b'))(\delta - r(x))/(\delta-r(b)) &
 \mbox{    if } r(x)>r(b),\end{array}\right.
\]
where $q(x)$, $r(x)$ and $r(b)$ refer to the representation of elements $V$
with offset $\delta$ and anchor $a$, and $r(b')$ refers to the
representation of elements $W$ with offset $\gamma$ and anchor $a'$. In other
words, $f$ is a convex mapping of the intervals $[a+k\delta,a+k\delta
  +r(b))$, $k=0,1,\dots,q(b)$, to the intervals
  $[a'+k\gamma,a'+k\gamma +r(b'))$, respectively, and of the intervals
    $[a+r(b)+k\delta,a+(k+1)\delta)$ to the intervals
      $[a'+r(b')+k\gamma,a'+(k+1)\gamma)$.
It is straightforward to verify that $f$ is a step-isometry
    at level $(\delta,\gamma)$, so any geometric $\delta$-graph and geometric $\gamma$ graph with vertex sets $V$
    and $W$, respectively, are isomorphic.
Another setting we consider is where $V$ and $W$ are disjoint unions of rational intervals for
which there exists a step-isometry between the endpoints of the intervals
of $V$ to the endpoints of the intervals of $W$.

\begin{theorem}
\label{thm:1Disometry}
Let $V$ and $W$ be two countable subsets of $\Rrr$, and let
$\delta,\gamma>0$. Let $F$ be a bijective step-isometry from $V$ to
$W$ at level $(\delta,\gamma)$.
If $G$ is a geometric $\delta$-graph with vertex set $V$ and $H$
is a geometric $\gamma$-graph with vertex set $W,$ then $G\cong H.$
\end{theorem}

\begin{proof}
Let $V=\{v_i:i\geq 0\}$ and $W=\{w_i:i\geq 0\}$, where
$w_i=F(v_i)$. As in the proof of
Theorem~\ref{thm:1D}, we inductively construct a sequence of pairs of
sets $(V_i,W_i)$ ($i\geq 0$) and
isomorphisms $f_i:G[V_i]\rightarrow H[W_i]$, so that for all $i\geq
1$, $v_i\in V_i$, $w_i\in W_i$, $V_{i}\subseteq V_{i+1}$ and
$W_i\subseteq W_{i+1}$, and $f_{i+1}$ extends $f_i$.
As an additional part of the induction hypothesis, we require that
$f_i$ satisfies the following three conditions.
\begin{enumerate}
\item For every $u,v\in V$, $r(u)\leq r(v)$ if and only if
  $r(f(u))\leq r(f(v)).$
\item For every $u,v\in V$, $r(u)\leq r(v)$ if and only if
  $r(f(u))\leq r(F(v)).$
\item For every $u\in V$, $q(u)=q(f(u))$.
\end{enumerate}
The first two conditions are those stated in
Lemma~\ref{lemma:isometry_conditions}, so this implies that $f_i$ is a
step-isometry at level $(\delta, \gamma).$ We can also conclude from
this lemma that for all $u,v\in V$, $r(u)\leq r(v)$ if and only if
$r(F(u))\leq r(F(v))$.

Let $V_0=\{ v_0\}$ and $W_0=\{ w_0\}$, and set
$f_0(v_0)=w_0$. Conditions (1) and (3) follow as in the proof of
Theorem~\ref{thm:1D}. Condition (2) follows from the fact that
$w_0=F(v_0=f(v_0)$. For the induction step, fix $i\geq 0$. we construct $f_{i+1}$ from
$f_{i}$ by first finding an image of $v_{i+1}$. In the following, $f$
refers to $f_{i}$, and $v=v_{i+1}$.

Let
\begin{eqnarray*}
M_a&=&
\{u:\, u\in V_i\mbox{ and }r(F(u))\leq r(F(v))\},\\
M_b&=&\{u:\, u\in V_i\mbox{ and } r(F(u))>r(F(v))\},
\end{eqnarray*}
and
\begin{eqnarray*}
a&=&\max\{x:\,x=r(f(u))\mbox{ or }x=r(F(u))\mbox{ where } u\in M_a\},\\
b&=&\min\{x:\,x=r(f(u))\mbox{ or }x=r(F(u))\mbox{ where } u\in M_b\}.\\
\end{eqnarray*}
We have that $a<b$, since the order of the representatives of vertices
in $V_i$ is preserved under $f$ and under $F$. (See the similar
argument in the proof of Theorem \ref{thm:1D}.)

In  order to maintain conditions (1) and (2) of the induction hypothesis, $r(f(v))$ should lie
in $[a,b)$, and because of condition (3), $q(f(v))$ must equal $q(v)$. Let $k=q(v)$, and
  consider the interval $I=(k\gamma+a,k\gamma+b)$. From the definition
  of $a$ and $b$ it follows that $F(v)\in I$.

The remainder of the
proof is now analogous to the proof of Theorem~\ref{thm:1D} and so
is only sketched here.
Let $x=F(v)$. We can show that $f(B_\delta(v)\cap V_i)\subseteq B_{\gamma}(x)$.
We can then invoke the $\gamma$-g.e.c. condition of $H$ to find a
vertex $z$ in $I$ which is correctly joined to the vertices in $W_i$
so that an isomorphism is maintained if we set $f(v)=w$. Finally, we
finish the induction step by going back and finding a suitable image
$f^{-1}(w_{i+1})$.
\end{proof}

Theorem~\ref{thm:1D} extends to $\mathbb{R}^n$ with $n>1,$ provided we
use the \emph{product metric}; that is, the metric derived from the $L_\infty$
norm, defined by:
\[
d(u,v)=\max\{ |v_i-u_i|\,:1\leq i\leq n\},
\]
where $u_i$ denotes the $i$-th component of $u$.
Hence, we obtain unique isotypes of
infinite random geometric graphs in all finite dimensions.  For the remainder
of the section, $n$ is a positive integer, and $d$ is assumed to be
the metric defined above.

\begin{theorem}
\label{thm:nD}
Consider the metric space $(\Rrr^n,d)$, where $d$ is the product metric
defined above.
Let $V$ and $W$ be two countable sets dense in $\Rrr^n$, and let
$\delta,\gamma>0$.
If $G$ is a geometric $\delta$-graph with vertex set $V$ and $H$ is a geometric $\gamma$-graph with vertex set $W$, then $G\cong H$. In particular, for
all choices of $V$ and $\delta$, there is unique isomorphism type of
geometric $\delta$-graphs in $(\Rrr^n,d)$, written $GR_n$.
\end{theorem}

Theorem~\ref{thm:nD} is sensitive to the choice of metric. In
Section~\ref{none}, we will show that the conclusion of Theorem~\ref{thm:nD} for $d$ the Euclidean metric
fails even for $n=2.$ The following provides
the key tool for our proof of Theorem~\ref{thm:nD}. As the proof is
straightforward generalization of
Lemma~\ref{lemma:isometry_conditions}, it is omitted.

\begin{lemma}
\label{lemma:isometry_conditions_Rn}
Let $V$ and $W$ be subsets of $\Rrr^n$ with the $L_\infty$-metric, let
$v_0\in V$ and $w_0\in W$, and
let $\delta$ and $\gamma$ be two non-negative real numbers.

Then a surjective function $f:V\rightarrow W$ is a step-isometry
at level $(\delta,\gamma)$ if the following two conditions hold for
all $u,v\in V$ and for all $i$, $1\leq i\leq n$:
\begin{enumerate}
\item $r(u_i)\leq
  r(v_i)$ if and only if $r(f(u)_i)\leq r(f(v)_i).$
\item $q(u_i)=q(f(u)_i)$,
\end{enumerate}
where the representation of the $i$-th coordinate of elements of $V$
has offset $\delta $ and anchor $(v_0)_i$ the representation of the
$i$-th coordinate of elements of $W$
has offset $\gamma $ and anchor $(w_0)_i$,
\end{lemma}

\begin{proof}[Proof of Theorem~\ref{thm:nD}]
Let $V=\{v_i:i\geq 0\}$ and $W=\{w_i:i\geq 0\}$. We inductively
construct a sequence of pairs of sets $(V_i,W_i)$ ($i\geq 0$) and
isomorphisms $f_i:G[V_i]\rightarrow H[W_i]$, so that for all $i\geq
1$, $v_i\in V_i$, $w_i\in W_i$, $V_{i}\subseteq V_{i+1}$ and
$W_i\subseteq W_{i+1}$, and $f_{i+1}$ extends $f_i$.
As an additional
induction hypothesis we require that $f_i$ satisfies conditions
(1) and (2) from Lemma~\ref{lemma:isometry_conditions_Rn}.

As in the proof of Theorem~\ref{thm:1D}, for the base case we take $V_0=\{ v_0\}$, $W_0=\{ w_0\}$
and $f_0(v_0)=w_0$.
For the induction step, fix $i\geq 0$. we construct $f_{i+1}$ from
$f_{i}$ by first finding an image of $v_{i+1}$.  In the following, $f$
refers to $f_{i}$, and $v=v_{i+1}$.

For all $j$, $1\leq j\leq n$, define
\begin{eqnarray*}
a_j&=&\max\{r(f(u)_j):u\in V_i\mbox{ and }r(u_j)\leq r(v_j)\},\\
b_j&=&\min\{r(f(u)_j):u\in V_i\mbox{ and } r(u_j)>r(v_j)\}.
\end{eqnarray*}

In  order to maintain the induction hypothesis, for all $j$,
$r(f(v))_j$ should lie in interval $[a_j,b_j)$, and $q(f(v)_j)$ should
  be equal to $q(v_j)$. Let $k_j=q(v_j)$, and consider the product set
\[
I=\prod_{1\leq j\leq n}(k_j\gamma+a_j,k_j\gamma+b_j).
\]
Any vertex in $I$ will qualify as a candidate for $f(v)$ so that $f$
satisfies conditions (1) and (2) from
Lemma~\ref{lemma:isometry_conditions_Rn}. The remainder of the proof
is analogous to that of Theorem~\ref{thm:1D}, and so is
omitted.
\end{proof}

A \emph{step-isometric isomorphism} is an isomorphism of graphs that
is a step-isometry. In the base step in the proof of
Theorem~\ref{thm:nD}, if we are given induced subgraphs $V_0$ and
$W_0$ such that $f_0:V_0 \rightarrow W_0$ is a step-isometric
isomorphism, then the rest of the proof follows as before. Hence, we
have the following corollary, which shows that the graphs
$GR_n$ act transitively on step-isometric isomorphic
induced subgraphs.

\begin{corollary}
Let $G$ and $H$ be finite induced subgraphs of $GR_n$ for some positive integer
$n$. A step-isometric isomorphism $f:G\rightarrow H$ extends to an
automorphism of $GR_n.$
\end{corollary}

Deleting a point from a dense set $V$ in $\Rrr^n$ gives
another dense set. Hence, we have the following
\emph{inexhaustibility} property.

\begin{corollary}
For all $n>0$ and vertices $x$ in $GR_n$, $GR_n-x \cong GR_n.$
\end{corollary}

We can combine Theorems~\ref{thm:1Disometry}~and~\ref{thm:nD} to obtain a
result about isomorphisms between graphs with vertex sets in $\Rrr^n$
if there exist a special type of map between the sets. Given a set
$V\subseteq \Rrr^n$, denote the $i$-th component set of $V$ as:
\[
V_i=\{ x_i:\, x\in V\}.
\]

\begin{theorem}
\label{thm:nDisometry}
Consider the metric space $(\Rrr^n,d)$, where $d$ is the product metric
defined above.
Let $V$ and $W$ be two countable sets in $\Rrr^n$, and let
$\delta,\gamma>0$. Assume that for all $1\leq i\leq n$, there exists a
step-isometry at level $(\delta,\gamma)$ from $V_i$ to $W_i$.
If $G$ is a geometric $\delta$-graph with vertex set $V$ and $H$ is a
a geometric $\gamma$ graph with vertex set $W,$ then $G\cong H$.
\end{theorem}

The proof is a straightforward generalization of the proofs of Theorems
\ref{thm:1Disometry} and \ref{thm:nD}, and is therefore omitted.

\section{Non-isomorphism results for Euclidean space}\label{none}

The choice of metric plays an important role in our isomorphism
results in Section~3. We demonstrate that there are non-isomorphic geometrically
e.c.\ graphs in the plane with the usual Euclidean metric (denoted by
$d$).

\begin{theorem}\label{non2}
Let $V$ be a countable set dense in $(\Rrr^2,d),$ and let $G$ and $H$ be two graphs generated by the model
$\mathrm{LARG}(V,p)$, where $0<p<1$. Then with probability 1, $G\ncong H.$
\end{theorem}

We have the following corollary, which is the antithesis of the
results in the previous section.

\begin{corollary}\label{noni}
Let $V$ be a countable set dense in $\Rrr^2$ equipped with the Euclidean metric, and $\delta >0$ fixed.
Then there exist infinitely many pair-wise non-isomorphic $\delta$-g.e.c.\ graphs with vertex
set $V.$
\end{corollary}

For the proof of Theorem~\ref{non2}, we rely on the following geometric lemma.

\begin{lemma}\label{lemnoni}
Let $V$ and $W$ be dense subsets of $\Rrr^2$ equipped with the
Euclidean metric. Then every step-isometry from $V$ to $W$ is an
isometry.
\end{lemma}

\begin{proof}
Assume for a contradiction that there is a step-isometry
$f:V\rightarrow W$ at level $(\delta,\gamma)$ that is not an
isometry. Without loss of generality, we assume that $\delta=\gamma=1$. For
each $u\in V$, let $u'=f(u)\in W$. Since $f$ is not an isometry, there
must exist points $x_1$ and $x_2$ so that $d(x_1,x_2)\not=
d(x_1',x_2')$.  Since $f$ is a bijection, we may assume, without loss
of generality, that $d(x_1,x_2) < d(x_1',x_2')$.

The proof follows by the following two claims. Given $x,y\in V$,
define \emph{the discrepancy of} $x,y$
as $$D(x,y)=|d(x,y)-d(x',y')|.$$ The discrepancy is a measure of the
error in the distance between pairs of points and their images under
$f$. Since $f$ is a step-isometry, we have that $D(x,y) <1$ for all $x,y \in V.$

\begin{claim}\label{claim1}
For every $\epsilon >0$, if there exist points $x_1,x_2 \in V$ so
that $$D(x_1,x_2)=\epsilon>0$$ and $d(x_1,x_2)>40$, then there exist
points $x_3,x_4 \in V$ so that $$D(x_3,x_4)>2\epsilon .$$
\end{claim}

\begin{claim}\label{claim2}
If there exist points $x_1,x_2\in V$ so
that $$D(x_1,x_2)=\epsilon>0,$$ then there exist points $x_3,x_4\in V$
so that $$D(x_3,x_4)>3/4\epsilon$$ and $d(x_3,x_4)>40$.
\end{claim}

To see how the lemma follows from the claims, note that by hypothesis,
there are two points of $V$ with discrepancy $\epsilon >0$. By
Claim~\ref{claim2}, there are two points of $V$ with discrepancy at
least $3/4\epsilon,$ and with distance at least $40$ apart. By
Claim~\ref{claim1} there are points with discrepancy at least
$3/2\epsilon$ apart. By induction, we obtain a sequence of pairs of
points $\{y_i,z_i\}$ of $V$ whose discrepancy equals
$(3/2)^i\epsilon$ which tends to infinity in $i$. In
particular, there are points $y,z$ of $V$ so that $D(y,z)>1$, which
contradicts the fact that $f$ is a step-isometry.

\smallskip

We now prove Claim~\ref{claim1}. We first define some constants that
will be useful in the proof. Let
\[
m=d(x_1,x_2)/2,\text{ and } k = \lfloor m\rfloor + 1.
\]
Choose
\[
0<\xi <\epsilon^2/(2k).
\]
Since $V$ is dense in $\Rrr^2$, we can find
points $x_3$ and $x_4$ in $V$ so that
\[
k-\xi< d(x_i,x_j)< k
\text{ for }
i=1,2 \text{ and } j=3,4.
\]
So $x_1,x_3,x_2,x_4$ are the vertices of a quadrilateral whose sides have
length between $k-\xi$ and $k$. See Figure~\ref{onee}.
\begin{figure}[h]\label{onee}
\begin{center}
\epsfig{figure=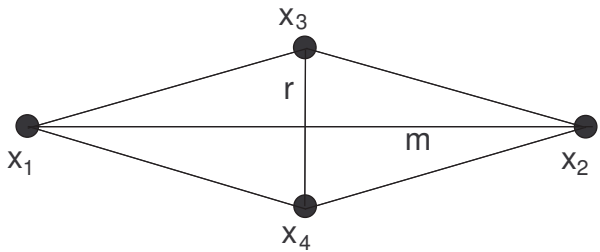}
\end{center}
\caption{The quadrilateral formed by $x_1,x_3,x_2,x_4$.}
\end{figure}

Let $r=d(x_3,x_4)/2$. The distance between $x_3$ and $x_4$ is smallest
when all sides of the quadrilateral equal $k-\xi$, so using the
Pythagorean theorem we have that
\begin{equation}
\label{eqn:r}
r^2\geq (k-\xi)^2-m^2 \geq k^2-2k\xi-m^2\geq k^2-\epsilon^2-m^2,
\end{equation}
where the last step follows from the choice of $\xi$.
(Note that the calculation above is only valid if we use the Euclidean metric.)

On the other hand, $x_3$ and $x_4$ are furthest when all sides of the
quadrilateral equal $k$. Since $k\leq m+2$ and $m>20$, we obtain
that
\begin{equation}
\label{eqn:half}
r^2\leq (m+2)^2-m^2=2m+1\leq m^2/4.
\end{equation}

Now consider the quadrilateral formed by the images
$x_1',x_2',x_3'$, and $x_4'$. Let
\[
m'=d(x_1',x_2')/2=m+\epsilon/2, \text{
  and }r'=d(x_3',x_4')/2.
\]
See Figure~3.
\begin{figure}[h!]
\begin{center}
\epsfig{figure=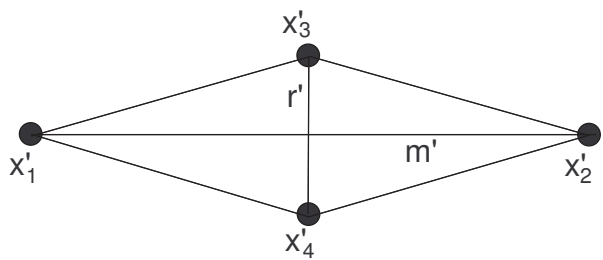}\label{fig2}
\end{center}
\caption{The quadrilateral formed by $x'_1,x'_3,x'_2,x'_4$.}
\end{figure}

Since $f$ is a step-isometry,
\[
d(x_i',x_j')<k \text{ for }i=1,2 \text{ and }
j=3,4.
\]
Now $d(x_3',x_4')$ is largest when the quadrilateral has all sides equal to
$k$. It follows that
\begin{eqnarray*}
(r')^2 &\leq & k^2-(m')^2\\
&=& k^2- (m+\epsilon/2)^2\\
&\leq & k^2-m^2-m\epsilon\\
& \leq & r^2-m\epsilon+\epsilon^2\\
&\leq & (r-\epsilon)^2
\end{eqnarray*}
where the third inequality follows from (\ref{eqn:r}), and the last
inequality follows from (\ref{eqn:half}). In particular, $r' \le r- \epsilon.$
As $$d(x_3',x_4')=2r'\leq 2r-2\epsilon=d(x_3,x_4)+2\epsilon,$$ the
proof of Claim~\ref{claim1} follows.

\smallskip

We now prove Claim~\ref{claim2}. Let $m=d(x_1,x_2),$ and assume
$m<40$. Let $k=40$. Choose $c>0$ so that
\begin{equation}\label{cc}
10c/3+c^2\leq k\epsilon/4\text{ and }c<(2-\sqrt{3})\epsilon/8.
\end{equation}
Further, choose points $x_3,x_4,x_5,x_6 \in V$ so
that
\begin{equation}
\label{upper}
k<d(x_i,x_j)< k+c,\text{ for }i,j\text{ equals }1,3\text{ or }1,5,\text{ or }2,4\text{ or }4,6,
\end{equation}
\begin{equation}
\label{lower}
k-c<d(x_i,x_j)< k,\text{ for }i,j\text{ equals }3,5\text{ or }4,6,
\end{equation}
and
\begin{equation}
\label{difference}
d(x_3,x_4)<d(x_5,x_6)<d(x_3,x_4)+c.
\end{equation}
The choice of such points
is possible since $V$ is dense. See Figure~4.
 \begin{figure}[h!]
\begin{center}
\epsfig{figure=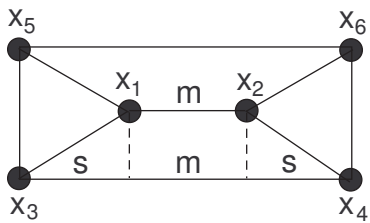}\label{fig3}
\end{center}
\caption{The figure formed by $x_i,$ $1\le i \le 6$.}
\end{figure}

Subject to the given constraints, the points $x_3$ and $x_4$ are
furthest apart when the distances achieve the upper bound of
(\ref{upper}) and the lower bound from (\ref{lower}), and when
$d(x_3,x_4)=d(x_5,x_6)$. Assume this to
be the case. Then $x_3,x_4,x_5,x_6$ form a rectangle, and the line
segments $x_1x_2$, $x_3x_4$ and $x_5x_6$ are parallel.

For $i=1,2$, let $y_i$ be the orthogonal projection of $x_i$ on the
line $x_3x_4$. Then $d(y_1,y_2)=d(x_1,x_2)=m$, and $d(x_3,y_1)=d(x_4,y_2)$; we will denote this
distance by $s$. See Figure~\ref{fig3}.

Hence, $d(x_3,x_4)=2s+m,$ and
$$
s^2 = (k+c)^2-1/4(k-c)^2 =3/4(k^2+10kc/3+c^2),
$$
and so
$$
d(x_3,x_4)=\sqrt{3(k^2+(10/3)kc+c^2)}+m.
$$
The expression above is based on the case where $x_3$ and $x_4$ are
furthest apart, and thus, it gives an upper bound for the general case.
Combining this upper bound with the condition on $c$ given by
(\ref{cc}), and with the assumption (\ref{difference}) we have
for $i,j$ equals $3,4$ and $5,6$ that
\begin{eqnarray*}
d(x_i,x_j)&
\leq & m+ c+ \sqrt{3(k^2+(10/3)kc+c^2)}\\
&\leq & m + c + \sqrt{3(k^2+ k\epsilon/4)}\\
& <& m+c+\sqrt{3}(k+\epsilon/8)\\
&< & m + \sqrt{3}k +\epsilon/4,
\end{eqnarray*}
where the first and last steps follow from (\ref{cc}).

We next consider the images of these points, and assume without loss
of generality that $d(x_3',x_4')>d(x_5',x_6')$. Since $f$ is a
step-isometry, $d(x_i',x_j')\geq k$ for $i,j$ equals $1,3$ or $1,5$
or $2,4$ or $2,6$, and $d(x_i',x_j')\leq k$ for $i,j$ equals $3,5$ or
$4,6$. Moreover, by assumption $d(x_1',x_2')=m+\epsilon $. Now $x_3'$ and $x_4'$
are closest together when $d(x_3',x_4')=d(x_5',x_6')$ and
$d(x_i',x_j')=k$ for all $i,j$ for which $|i-j|$ is even. As in the
previous case, the line segments $x_5'x_6'$, $x_1'x_2'$, and
$x_3'x_4'$ are parallel, and $x_5'x_6'x_4'x_3'$ is a rectangle. Under
these assumptions
we can compute  $d(x_3',x_4')$ similarly to the computation for $d(x_3,x_4)$, and obtain that
$$
d(x_3',x_4')= 2(k\sqrt{3}/2)+(m+\epsilon).
$$
Since our assumptions hold for the case where $x_3'$ and $x_4'$ are
closest together, we have in general, for $i,j$ equal to $3,4$ and
$5,6$, that
$$
d(x_i',x_j')\geq k\sqrt{3}+m+\epsilon\geq d(x_i,x_j)+3\epsilon/4.\qedhere
$$
\end{proof}

A direct consequence of this lemma is the existence of many
non-isomorphic g.e.c. graphs with vertex sets dense in $\Rrr^2$,
equipped with the Euclidean metric. A set $V$ in $\Rrr^2$ is $\delta$-\emph{free} if no pair of points in $V$ are distance $\delta$ apart. For instance, one may consider
$\delta = 1$, $V$ to be the set of all rational points in $\Rrr^2$ which are $\delta$-free, and $W = V \cup
\{ (2^{1/4},0) \}$ (which is also $\delta$-free). It is
straightforward to see there is no isometry from $V$ onto $W$. Hence,
a $1$-g.e.c.\ graph on $V$ cannot be isomorphic to a $1$-g.e.c.\ graph
on $W.$

\vspace{.1in}
In the following proof, we use the notation $\mathbb{P}(A)$ for the probability of an event $A.$

\begin{proof}[Proof of Theorem~\ref{non2}]
An enumeration $\{ v_i:i\in \Nnn^+\}$ of $V$ is \emph{good} if
$d(v_i,v_{i+1})<\delta$ for all $i\in \Nnn^+$ and $\{v_1,v_2,v_3\}$ are
not collinear. We claim that a countable set $V$ dense in $\Rrr$ has a good enumeration.
For a positive integer $n$, we call $\{ v_i:1\le i \le
n\}$ a \emph{partial good enumeration} of $V.$ We prove the claim
by constructing a chain of partial good enumerations
by induction. Using the density of $V,$ choose three points
$\{v_1,v_2,v_3\}$ that are not collinear, so that each are within
$\delta$ of each other. Let $V_1 =\{v_1,v_2,v_3\}.$
Enumerate $V \setminus
\{v_1,v_2,v_3\}$ as $\{u_i: i\ge 2\}.$ Starting from $V_1$, we
inductively construct a chain of partial good enumerations $V_n$,
$n\geq 1$, so that for $n\geq 2$, $V_n$ contains $\{u_i: 2\le i
\le n\}.$

We now want to form $V_{n+1}$ by adding $u=u_{n+1}$. If $u\in V_n$,
then let $V_{n+1}=V_n$. Assume without loss of generality that
$u \not\in V_n.$ Let $N=|V_n|$. If $d(v_N ,u) < \delta,$
then let $v_{N+1}=u$ and add it to $V_n$ to form $V_{n+1}.$ Otherwise, by the
density of $V,$ choose a shortest finite path $P=p_0,\dots, p_\ell$ of
points of $V\setminus V_n$
starting at $v_N=p_0$ and ending at $u=p_\ell$ so that two consecutive points
in the path are distance at most $\delta$. Then add the vertices of
$P$ to $V_n$ to form $V_{n+1}$ and enumerate them so that $v_{N+i}=p_i$ for
$i=0,1,\dots,\ell$.
Taking the limit of this chain,
$\bigcup_{n\ge 1} V_n$ is a
good enumeration of $V$, which proves the claim.

Let $V=\{ v_i:i\geq 1\}$ be a good enumeration of
$V,$ and for any $n$, let $V_n = \{ v_i: 1\le i \le n \}.$ Let $G$ and $H$ be as
stated. We say that two pairs $\{ v,w\}$ and $\{ v',w'\}$ of vertices
are \emph{compatible} if $\{ v,w\}$ are adjacent in $G$ and $\{
v',w'\}$ are adjacent in $H$ or $\{ v,w\}$ are non-adjacent in
$G$ and $\{ v',w'\}$ are non-adjacent in $H$. For two pairs $\{ v,w\}$
and $\{ v',w'\}$ such that $d(v,w)=d(v',w')$, the probability that they
are compatible equals
\[
p^*=\left\{\begin{array}{ll}
p^2+(1-p)^2 & \text{if }d(v,w)<\delta\text{' and}\\
1 &\text{otherwise.}
\end{array}\right.
\]

By Corollary~\ref{useful} and Lemma~\ref{lemnoni}, any isomorphism
between subgraphs of $G$ and $H$ must be an isometry.  The images of
three points in $\Rrr^2$ that are not collinear determine the
isometry. Let $A_n$ be the event that there exists a partial
isomorphism $f$ from $G[V_n]$ into $H$ so that
$f(\{v_1,v_2,v_3\})\subseteq V_n$, and let
$$
A^*_n=\bigcap_{\nu \geq
  n} A_{\nu}.
$$
Note that $A^*_{n}\subseteq A^*_{n+1}$ for all $n$.

Next, we estimate the probability of $A^*_n$. Note first that
$\mathbb{P}(A^*_n)\leq \mathbb{P}(A_\nu)$ for all $\nu\geq n$. For any tuple
$(u_1,u_2,u_3)$ of three distinct vertices in $V_n$, let
$C_n(u_1,u_2,u_3)$  be the event that there exists a partial
isomorphism $f$ from $G[V_n]$ to $H$ so that $f(v_i)=u_i$ for
$i=1,2,3$. Since the images of three points that are not collinear
determine the isometry $f$, if $C_n$ happens then all pairs
$(v_i,v_{i+1})$ and $(f(v_i),f(v_{i+1}))$ must be compatible, for
$1\leq i<n$. Thus, $$\mathbb{P}(C_n(u_1,u_2,u_3))\leq (p^*)^{n-1}.$$ Now
\[
A_n=\bigcup_{u_1,u_2,u_3\in V_n} C_n(u_1,u_2,u_3),
\]
so for $n\ge 3$ we have that $\mathbb{P}(A_n)\leq n^3 (p^*)^{n-1}$, and
\[
\mathbb{P}(A^*_n)\leq \inf\{ \nu^3 (p^*)^{\nu-1} : \nu\geq n\} =0.
\]

If $B$ is the event that $G$ and $H$ are
isomorphic, then $$B\subseteq \bigcup_{n\in \Nnn^+} A^*_n.$$ Since
the union of countably many sets of measure zero has measure zero, we
conclude that $\mathbb{P}(B)=0$, and thus, with probability $1,$ $G\ncong H.$ \end{proof}

\end{document}